\RequirePackage{fix-cm}
\documentclass[fleqn,reqno,smallextended]{svjour3}       
\smartqed  
\usepackage{graphicx}
\usepackage{amsfonts,amssymb,amscd,amsmath,enumerate,verbatim,calc}
\usepackage[colorlinks,citecolor=blue]{hyperref}
\usepackage{cite}

\spnewtheorem{result}{Result}{\bf}{\it}

\DeclareMathOperator{\grad}{grad}

\DeclareMathOperator*{\Ric}{Ric}
\DeclareMathOperator*{\diver}{div}

\newcommand{\co}{\nabla}
\newcommand{\D}{\partial}
\newcommand{\di}{\mathcal{D}}
\newcommand{\n}{{M_1 }\times_{f}M_2}

\begin{document}

\title{On cosmological constant of Generalized Robertson-Walker space-times 
}


\author{Morteza Faghfouri \and Ali Haji-Badali\and Fatemeh Gholami
}

\authorrunning{M. Faghfouri \and    A. Haji Badali   \and F. Gholami
} 
\institute{Morteza Faghfouri \at Faculty of Mathematics, University of Tabriz, Tabriz, Iran\\              \email{faghfouri@tabrizu.ac.ir}           \and        Ali Haji Badali  \at    Department of Mathematics, Faculty of Basic Sciences,University of Bonab, Bonab, Iran\\               \email{haji.badali@bonabu.ac.ir}
\and           Fatemeh Gholami \at            Department of Mathematics, Faculty of Basic Sciences,   University of Bonab, Bonab, Iran\\               \email{fateme.gholami@bonabu.ac.ir}
}

\maketitle

\begin{abstract}
We study Einstein's
equation in $(m+n)D$  and $(1+n)D$ warped spaces $(\bar{M},\bar{g})$
 and classify all such spaces satisfying
Einstein equations $\bar{G}=-\bar{\Lambda}\bar{g}$. We show that the
warping function 
 not only can determine the cosmological
constant $\bar{\Lambda}$ but also it can determine the cosmological
constant ${\Lambda}$ appearing in the induced Einstein equations
${G}=-{\Lambda}{h}$ on $(M_2, h)$. Moreover, we discuss on the
origin of the $4D$ cosmological constant as an emergent effect of
higher dimensional warped spaces.

\end{abstract}

\section{Introduction}
One of the most fruitful generalizations of the direct product of two pseudo-Riemannian manifolds is the warped product defined by  Bishop and O'Neill in Ref.\cite{bishop.oneill:}. The notion of warped products plays very important
roles in differential geometry as well as in mathematical physics, especially
in general relativity.

 Many basic solutions of the Einstein field equations are warped products.
For instance, both Schwarzschild and Robertson-Walker models
in general relativity are warped products. Schwarzschild spacetime is the best relativistic model which describes the outer spacetime around a massive star or a black hole and the Robertson-Walker model describes a simply connected
homogeneous isotropic expanding or contracting universe.

 In Lorentzian geometry, it was first noticed that some well known solutions
to Einstein's field equations can be expressed in terms of warped
products \cite{Beem.Ehrlich:GlobalLorentzianGeometry} and afterwards
Lorentzian warped products have been used to obtain more solutions
to Einstein's field equations. On the other hand, first attempts for
finding a static solution of the field equations of general
relativity, applied for the cosmology, was done by Einstein who
introduced the cosmological constant \cite{Einstein}. Since then the
cosmological constant has played a very important role in the
context of cosmology. However, the problem of origin and the large
disagreement between the theoretical  prediction, based on the
quantum field theory calculations, and the observational bound on
the value of this cosmological term has remained as the well known
``cosmological constant problem'' \cite{CCP,Darabi:AquantumCosmologyAndDiscontinuous}. Here, we do not
propose a solution for this problem, but we aim to present a new
origin for the cosmological constant, based on the warped product of
two pseudo-Riemannian manifolds. In this work, we study the
generalized Robertson-Walker warped spacetimes with a cosmological
constant.

Generalized Robertson-Walker spacetime  models
 and standard static spacetime models are two well known solutions to Einstein's
field equations which can be expressed as Lorentzian warped products.
 Bejancu {\it et al} in Ref.\cite{Bejancu:Classificationof5dwarpedspaceswithcosmologicalconstant}, by using the extrinsic curvature of the horizontal distribution, obtained  the classification of all spaces $(\bar{M},\bar{g})$ satisfying Einstein equations $\bar{G}=-\bar{\Lambda}\bar{g},$ where  $(\bar{M},\bar{g})$ is a $5D$ warped space defined by $4D$ spacetime $(M,g)$ and the warped function. They described  all the exact solutions for the warped metric by means of $4D$ exact solutions.
We generalize the formalism of Bejancu {\it et al} in $5D$ warped
space to $(m+n)D$  and $(1+n)D$ warped spaces. The main motivation
of this work is to describe the cosmological constant in Einstein's
equation as a resultant of the geometry of warped spaces.

In  section 2, we recall some necessary details  on warped product of pseudo-Riemannian manifolds, Ricci tensor,  scalar  curvature and  Einstein gravitational equation  on warped space. In section 3, we find necessary and sufficient conditions for the validity of Einstein equations on warped product space. Finally, in section 4, we construct five classes of exact solutions of Einstein equation for generalized Robertson-Walker spacetime.

\section{Preliminaries}
Let $\psi$ be a smooth  function on a pseudo-Riemannian $n$-manifold $(M,g)$. Then  the Hessian tensor field of $\psi$ is given by
$H^\psi(X,Y)=XY\psi-(\co_XY)\psi,$
and
the Laplacian of $\psi$ is given by
$\Delta\psi=\text{trace }(H^\psi),$ or equivalently, $\Delta = \diver(\grad),$  where $\co$, $\diver$ and $\grad$ are  Levi-Civita
connection of $M$,  the divergence
and  the gradient operators, respectively.  (see p. 85 of \cite{oneil:book}).  Furthermore, we will frequently use the notation
$\Vert\grad f\Vert^2 = g(\grad f, \grad f ).$
In terms of a coordinate system $(x^1,\ldots,x^n)$, we have
\begin{align}
\Delta\psi=g^{ij}\left(\dfrac{\D^2 \psi}{\D x^i\D x^j}-\Gamma_{ij}^k\dfrac{\D \psi}{\D x^k}\right),
\end{align}
where $\Gamma _{ij}^k$ are the Christoffel symbols on $(M,g)$ given by
$\Gamma_{ij}^k=\frac{1}{2}g^{kl}\left(\frac{\D g_{il}}{\D x^j}+\frac{\D g_{jl}}{\D x^i}-\frac{\D g_{ij}}{\D x^l}\right).$

We state the following  Lemmas for later uses.
\begin{lemma}[Ref. \cite{chen:2011pseudo}]\label{lemma:hopf1}
Every harmonic function on a compact Riemannian manifold
is  constant.
\end{lemma}
\begin{lemma}[Hopf's lemma. in Ref. \cite{chen:2011pseudo} ]\label{lemma:hopf2}
Let $M$ be a compact Riemannian manifold. If $\psi$ is a
differentiable function on $M$ such that $\Delta\psi \geq0$  everywhere on $M$ (or $\Delta\psi \leq0$
everywhere on $M$), then $\psi$ is a constant function.
\end{lemma}

Let $(M_1,g)$ and $(M_2,h)$ are two pseudo-Riemannian manifolds and $f$
be a positive smooth function on $ M_1$. Then the warped product
$\bar{M}=M_1\times_f M_2$ is the product manifold $M_1\times M_2$
endowed with the pseudo-Riemannian metric $\bar{g} = \pi^*g +
(f\circ\pi)^2\sigma^*h,$
where $\pi$ and $\sigma$ are the projections of $M_1\times M_2$ onto
$M_1$ and $M_2$ respectively. The function $f$ is called a warping function and also $(M_1, g)$ and $(M_2,h)$ is called  a base
manifold and a fiber manifold, respectively\cite{oneil:book}.

The  warped product $(\bar{M}, \bar{g})$ is a Lorentzian  warped product if $(M_2,h)$
are Riemannian and either $(M_1,g)$ is Lorentzian or else
$(M_1,g)$ is a one-dimensional manifold with a negative definite metric $-dt^2$, Ref.
\cite{Beem.Ehrlich.Easley:GlobalLorentzianGeometry}.
We recall the de sitter space  with cosmological constant $\Lambda>0$, and
\begin{align}
ds^2=-dt^2+e^{2(\frac{\Lambda}{3})^{\frac{1}{2}}t}\left(dr^2+r^2(d\theta^2+\sin^2\theta d\phi^2)\right),
\end{align}
where $(r,\theta,\phi)$ are spherical coordinates\cite{Misner.thorne:gravitation}. The de sitter space is an example of the Lorentzian warped product.

  One of the main properties of
$\bar{M}$ is that the two factors $(M_1,g)$ and $(M_2,h)$ are
orthogonal with respect to $\bar{g}$.
 For a vector field $X$ on $M_1$, the lift of $X$  to $\n$ is the vector field $\tilde{X}$ whose value at each $(p,q)$ is the lift
 $X_p$ to $(p,q)$. Thus the lift of $X$ is the unique vector field on $\n$ that is $\pi_1$-related to $X$ and $\pi_2$-related to
 the zero vector field on $M_2$.
For a warped product $\n$, let $\di_i$ denotes the distribution obtained from the vectors tangent to the horizontal lifts of
$M_i$.

Let $\bar{M}=({M_1 }\times_{f}M_2,\bar{g})$  be a warped product of pseudo-Riemannian manifolds  $(M_1,g)$ and $(M_2,h)$ with
metric $\bar{g}=g\times_f h$. If $X, Y, Z\in\di_1$ and $V, W, U\in\di_2$, then
\begin{align}
&\bar{\co}_XY=\co^1_XY,\\
&\bar{\co}_XV=\bar{\co}_VX=\frac{X(f)}{f}V,\\
&\bar{\co}_VW=\co^2_VW-\frac{g(V,W)}{f}\grad f,
\end{align}
where $\bar{\co}, \co^1$ and $ \co^2 $ are   the Levi-Civita connection of metrics $\bar{g}, g$ and $h,$ respectively.
\begin{align}
&\bar{R}(X,Y)Z=R^1(X,Y)Z.\\
&\bar{R}(V,X)Y=-\frac{H^f_1(X,Y)}{f}V, \text{ where } H^f_1 \text{ is the Hessian of } f.\\
&\bar{R}(X,Y)W=\bar{R}(V,W)X=0.\\
&\bar{R}(X,V)W=g(V,W)/f\co^1_X(\grad f).\\
&\bar{R}(V,W)U=R^2(V,W)U-\Vert\grad f\Vert^2/f^2(g(V,U)W-g(W,U)V),
\end{align}
where $\bar{R}, R^1$ and $ R^2 $ are   the curvature tensor  of metrics  $\bar{g}, g$ and $h,$ respectively.
\begin{align}
&\bar{\Ric}(X,Y)=\Ric^{M_1}(X,Y)-\frac{n}{f}H^f_{M_1}(X,Y).\\
&\bar{\Ric}(X,V)=0.\\
&\bar{\Ric}(V,W)=\Ric^{M_2}(V,W)-g(V,W) \left(\frac{\Delta
f}{f}+(n-1)\frac{\|\grad f\|^2}{f^2}\right),
\end{align}
where $\Ric, \Ric^{M_1}$ and $ \Ric^{M_2} $ are   the Ricci tensor  of metrics  $\bar{g}, g$ and $h,$ respectively.
\begin{align}
\bar{S}=S^{M_1}+\frac{S^{M_2}}{f^2}- 2n\frac{\Delta
 f}{f}-n(n-1)\frac{\|\grad f\|^2}{f^2},\label{eq:Scalar}
 \end{align}
where $S^{M_1}$ is scalar curvature of $(M_1,g)$ and $S^{M_2}$ is scalar curvature of
 $(M_2,h)$\cite{oneil:book}.

We use the Einstein convention, that is, repeated indices with one upper index and one lower index denote summation over their
range. If not stated otherwise, throughout the paper we use the following ranges for indices:  $ i, j, k, ...\in \{1,...,m\}$;
$\alpha, \beta, ...\in \{m+1,...,m+n\}$;
  $a,b,c,... \in \{1,...,n+m\}$. In what follows we take $( x^i,x^\alpha)$ as a coordinate system on $M_1\times M_2$, where
  $(x^i)$  and
$(x^\alpha)$ are the local coordinates on $ M_1$ and $ M_2$,
respectively.


Suppose that $\bar{g}$ is a pseudo-Riemannian
metric on $\bar{M}$
defined by $ \bar{g} =g\times_f h$
  given by its local components:
 \begin{align}
 \bar{g}_{ij}(x^a)&=g_{ij}(x^k),\\
 \bar{g}_{i\alpha}(x^a)&=0,\\
   \bar{g}_{\alpha\beta}(x^a)&= f^2(x^k)h_{\alpha\beta}(x^\mu), \label{eq:metric2}
 \end{align}
 where $g_{ij}(x^k)$ are the local components of $g$ and $h_{\alpha\beta}(x^\mu)$ are the local components of $h$.
Let $\bar{M}=M_1\times_f M_2$ be a warped product manifold.
Then
\begin{align}
\bar{R}_{ij}&=R_{ij}-\frac{n}{f}H^f_{ij},\label{eq:Ricij}\\
\bar{R}_{i\alpha}&=0,\\
\bar{R}_{\alpha\beta}&={R}_{\alpha\beta}-\left(\frac{\Delta
f}{f}+(n-1)\frac{\|\grad f\|^2}{f^2}\right)\bar{g}_{\alpha\beta},\label{eq:Ric2ij}
\end{align}
where $\displaystyle R_{ij}=\Ric^{M_1}(\partial_i,\partial_j)$ are the local components of
 Ricci tensor of $ (M_1,g) $ and $R_{\alpha\beta}=\displaystyle\Ric^{M_2}(\partial_\alpha,\partial_\beta)$
are the local components of Ricci tensor  of $(M_2,h)$.

We denote by $\bar{G}$ the Einstein gravitational tensor field of
$(\bar{M}, \bar{g})$, that is, we have,
\begin{align}
\bar{G} = \bar{\Ric} - \frac{1}{2}\bar{S}\bar{g},\label{eq:Einstein}
\end{align}
where $\bar{\Ric}$ and $\bar{S}$ are Ricci tensor and scaler curvature of $\bar{M}$, respectively.

\begin{pro}

Let $\bar{G}$ be the Einstein gravitational tensor field of
$(\bar{M},\bar{g})$, then we have following equations:\\
\begin{align}
\bar{G}_{ij}&= G_{ij}- \frac{n}{f}H^f_{ij}-
\frac{1}{2}\left(\frac{S^{M_2}}{f^2}- 2n\frac{\Delta
 f}{f}-n(n-1)\frac{\|\grad f\|^2}{f^2}\right)g_{ij},\label{eq:ein1}\\
\bar{G}_{\alpha\beta}&=G_{\alpha\beta}-f^2\left(\frac{\Delta
f}{f}(1-n)+\frac{1}{2}S^{M_1}+
(n-1)(\frac{2-n}{2})\frac{\|\grad f\|^2}{f^2}\right)h_{\alpha\beta},\label{eq:ein2}\\
\bar{G}_{i\alpha}&=0,\label{eq:ein3}
\end{align}
where $G_{ij}$ and $G_{\alpha\beta }$ are the local components of the Einstein
gravitational tensor field  of $(M_1, g)$ and  $(M_2, h)$, respectively.
\end{pro}
\begin{proof}

By using \eqref{eq:Einstein}, \eqref{eq:Ricij},
\eqref{eq:metric2} and \eqref{eq:Scalar}, we obtain\eqref{eq:ein1}, \eqref{eq:ein2} and \eqref{eq:ein3}.
\end{proof}
\section{Einstein field equations with cosmological constant}
Let $(\bar{M},\bar{g})$ be the warped space and $f$ be  the warped function. Suppose that the Einstein gravitational tensor field
$\bar{G}$ of $(\bar{M},\bar{g})$ satisfies the
Einstein equations with cosmological  constant $\bar{\Lambda}$ as
\begin{align}
\bar{G}=-\bar{\Lambda}\bar{g}.\label{6.1}
\end{align}
First, we prove the following theorem.
\begin{theorem}\label{1}
The Einstein equations on $(\bar{M},\bar{g})$ with cosmological constant  $\bar{\Lambda}$ are equivalent with the following
equations
\begin{align}
&-\bar{\Lambda}=\left(\frac{m+n-2}{2m}\right)\left (n\frac{\Delta f}{f}-S^{M_1}\right),\label{115}\\
&G_{\alpha\beta}=f^2(1-\frac{n}{2})\left(\frac{\Delta
f}{f}(1-\frac{n}{m})+\frac{S^{M_1}}{m}+(n-1)\frac{\|\grad
f\|^2}{f^2}\right)h_{\alpha\beta}.\label{116}\end{align}
Moreover, we have
\begin{align}\label{25}
&R_{\alpha\beta}=f^2\left(\frac{\Delta
f}{f}(1-\frac{n}{m})+\frac{S^{M_1}}{m}+(n-1)\frac{\|\grad
f\|^2}{f^2}\right)h_{\alpha\beta}.
\end{align}
\end{theorem}
\begin{proof}
Using \eqref{eq:ein1} and  Einstein equation $\bar{G}=-\bar{\Lambda}\bar{g}$, we have
\begin{align}
G_{ij}- \frac{n}{f}H^f_{ij}-
\frac{1}{2}g_{ij}\left(\frac{S^{M_2}}{f^2}- 2n\frac{\Delta
 f}{f}-n(n-1)\frac{\|\grad f\|^2}{f^2}\right)=-\bar{\Lambda}g_{ij}.\label{13}
\end{align}
Contracting \eqref{13} with $g^{ij},$ and  consider $G_{ij}g^{ij}=S^{M_1}-\frac{m}{2}S^{M_1}$ and $H^f_{ij}g^{ij}=\Delta f,$
we obtain
\begin{align}\bar{\Lambda}=-\left(S^{M_1}(\frac{1}{m}-\frac{1}{2})+n(1-\frac{1}{m})\frac{\Delta
f}{f}-\frac{1}{2}\frac{S^{M_2}}{f^2}+\frac{n}{2}(n-1)\frac{\|\grad
f\|^2}{f^2}\right).\label{14}\end{align}
Then  using \eqref{eq:ein2}, we obtain
\begin{align}
G_{\alpha\beta}=f^2h_{\alpha\beta}\left((1-n)\frac{\Delta f}{f}
+(n-1)(1-\frac{n}{2})\frac{\|\grad
f\|^2}{f^2}+\frac{1}{2}S^{M_1}-\bar{\Lambda}\right).\label{152}
\end{align}
By putting  $\bar{\Lambda}$ from \eqref{14} in \eqref{152}, we  have
\begin{align}
G_{\alpha\beta}=f^2h_{\alpha\beta}\left((1-\frac{n}{m})\frac{\Delta f}{f}+(n-1)\frac{\|\grad
f\|^2}{f^2}+\frac{S^{M_1}}{m}-\frac{1}{2}\frac{S^{M_2}}{f^2}\right).\label{15}
\end{align}
Now, contracting \eqref{15} with $ h^{\alpha\beta}$ we obtain
\begin{align}
\frac{S^{M_2}}{f^2}=n(1-\frac{n}{m})\frac{\Delta f}{f}+n(n-1)\frac{\|\grad
f\|^2}{f^2}+n\frac{S^{M_1}}{m}.\label{16}
\end{align}
By using \eqref{14} and \eqref{16},   we obtain
\begin{align}
-\bar{\Lambda}=\left(\frac{m+n-2}{2m}\right)\left (n\frac{\Delta f}{f}-S^{M_1}\right).
\end{align}
We now use \eqref{16} in \eqref{15} and obtain
\begin{align}
G_{\alpha\beta}=f^2h_{\alpha\beta}(1-\frac{n}{2})\left((1-\frac{n}{m})\frac{\Delta f}{f}+(n-1)\frac{\|\grad
f\|^2}{f^2}+\frac{S^{M_1}}{m}\right).
\end{align}
Now, from the above and by using \eqref{16} and \eqref{eq:Einstein}, we obtained \eqref{25}.
\end{proof}

\begin{pro}
The Einstein equations $\bar{G}=-\bar{\Lambda}\bar{g}$ on
$(\bar{M},\bar{g})$ with cosmological constant $\bar{\Lambda}$
induces the Einstein equations $G_{\alpha\beta} = -\Lambda
h_{\alpha\beta}$ on $(M_2, h_{\alpha\beta})$, where the cosmological
constant $\Lambda$ is given by
\begin{align}\label{24}
\Lambda=-f^2(1-\frac{n}{2})\left(\frac{\Delta
f}{f}(1-\frac{n}{m})+\frac{S^{M_1}}{m}+(n-1)\frac{\|\grad
f\|^2}{f^2}\right).\end{align}
\end{pro}
\begin{proof}
By using  \eqref{116}  and Einstein equations $G_{\alpha\beta} =
-\Lambda h_{\alpha\beta}$, we can  obtain \eqref{24}. Also, by using \eqref{25} and Theorem 3.3 of \cite[page 38]{yano:structures.on.manifolds},
\begin{align}
f^2\left(\frac{\Delta
f}{f}(1-\frac{n}{m})+\frac{S^{M_1}}{m}+(n-1)\frac{\|\grad
f\|^2}{f^2}\right)
\end{align}
must be a constant, and therefore $(M_2,h)$ is an Einstein manifold.
\end{proof}
Now using  \eqref{115}  and Lemma \ref{lemma:hopf1} and \ref{lemma:hopf2}, we deduce the following corollaries.
\begin{cor}
The warping function $f$ is an eigenfunction of the Laplacian operator $\Delta$ with
eigenvalue $\dfrac{2m\bar{\Lambda}+(m+n-2)S^{M_1}}{n(m+n-2)}.$
\end{cor}

\begin{cor}
 Let $\bar{M}=M_1\times_f M_2$ and satisfy an Einstein equation with a cosmological constant $\bar{\Lambda}$. If $M_1$ is a compact Riemannian manifold
 and Ricci flat  then $f$ is constant.
\end{cor}
\begin{cor}
 Let $\bar{M}=M_1\times_f M_2$ which satisfy an Einstein equation with a cosmological constant $\bar{\Lambda}=0$. If $M_1$ is   Ricci flat then $f$ is harmonic.
\end{cor}

\section{Generalized Robertson-Walker spacetimes}
An $(n+1)$-dimensional generalized Robertson-Walker (GRW) spacetime with $n>1$ is a
Lorentzian manifold which is a warped product manifold $\bar{M}=I\times_f M$ of an open interval
$I$ of the real line $\mathbb{R}$ and a Riemannian $n$-manifold $(M, g)$ endowed with the
Lorentzian metric
\begin{align}
\bar{g}=-\pi^*(dt^2)+f(t)^2\sigma^*(g),
\end{align}
where $\pi$ and $\sigma$ denote the projections onto $I$ and $M$, respectively, and $f$ is a positive
smooth function on $I$. In a classical Robertson-Walker (RW) spacetime, the fiber is
three dimensional and of constant sectional curvature, and the warping function $f$ is
arbitrary\cite{chen:ASimpleCharacteration}.
Such spaces include the Einstein-de Sitter space, the Friedman cosmological models, and the de Sitter space.

 The following formula can be directly obtained from the previous result and noting that
on a generalized Robertson-Walker spacetime $\grad_I f = -f',\Vert \grad_If\Vert^2_I =-f'^2, g(\frac{\D}{\D t},
\frac{\D}{\D t})=-1,H^f(\frac{\D}{\D t}, \frac{\D}{\D t})=f''$
 and $\Delta_I f=-f''.$ We denote the usual derivative
on the real interval I by the prime notation (i.e.,$'$)  from now
on.

 Putting $m=1$ in Theorem~\ref{1} and by
using \eqref{115} and \eqref{116}, we have
\begin{align}
&\bar{\Lambda}=-\frac{1}{8}n(n-1)(B^2+2B'),\label{126}\\
&G_{\alpha\beta}=-\frac{1}{4}(n-1)(n-2)f^2B'h_{\alpha\beta},\label{127}
\end{align}
 where $B=2\frac{f'}{f}.$ Moreover, Einstein Equation \eqref{127} on $(\bar{M},\bar{g})$ with
cosmological constant $\bar{\Lambda}$ induce the Einstein Equation
$G=-\Lambda g$ on $(M,g)$ where the cosmological constant $\Lambda$
is given by
\begin{align}
\Lambda=\frac{1}{4}(n-1)(n-2)f^2B',
\end{align}
and is constant.
Next, we assume that $\bar{\Lambda}$ given by \eqref{126} is a constant. Thus, we must solve the differential equation $-\frac{1}{8}n(n-1)(B^2+2B')=k,$ where $k$ is a constant. We can state the following classification theorem.

\begin{theorem}
Let $(\bar{M},\bar{g})$ be  an $(n+1)$-dimensional generalized Robertson-Walker (GRW) spacetime with $n>1$. We have
\begin{enumerate}
\item The cosmological constant of $(\bar{M},\bar{g})$ is given by
\begin{align}
\bar{\Lambda}=-\frac{n(n-1)}{2L^2}, L\neq0,
\end{align}
and $\bar{g}$ is given by one of the expressions
\begin{align}
ds^2&=\frac{1}{k}e^{2t/L}g_{\alpha\beta}(x^\mu)dx^\alpha dx^\beta-dt^2,\label{6.26a}\\
ds^2&=\frac{1}{k}\left(\cosh\frac{b+t}{L}\right)^2g_{\alpha\beta}(x^\mu)dx^\alpha dx^\beta-dt^2,\label{6.26b}\\
ds^2&=\frac{1}{k}\left(\sinh\frac{b+t}{L}\right)^2g_{\alpha\beta}(x^\mu)dx^\alpha dx^\beta-dt^2, t\neq-b,\label{6.26c}
\end{align}
where $k>0,$ and $b\in\mathbb{R}.$ Moreover, the spacetime $(M,g)$ must be Ricci flat in case of the metric
(\ref{6.26a}), and an Einstein space in cases of both metrics (\ref{6.26b}) and (\ref{6.26c}), with cosmological
constants
\begin{align}
\Lambda=-\frac{(n-1)(n-2)}{2kL^2},
\end{align}
and
\begin{align}
\Lambda=\frac{(n-1)(n-2)}{2kL^2},
\end{align}
respectively.
\item
The cosmological constant of $(\bar{M} ,\bar{g})$ and $\bar{g}$ are given by
\begin{align}
\bar{\Lambda}=\frac{n(n-1)}{2L^2}, L\neq0,
\end{align}
and
\begin{align}\begin{split}
ds^2=\frac{1}{k}\left(\cos\frac{b+t}{L}\right)^2g_{\alpha\beta}(x^\mu)dx^\alpha dx^\beta-dt^2,\label{6.30}\\
t\neq b+L\pi(2h+1)/2,  h\in\mathbb{Z},
\end{split}
\end{align}
where $k > 0,$ and $b\in\mathbb{R}$. In this case, $(M, g)$ is an Einstein space with cosmological constant
\begin{align}
\Lambda=\frac{(n-1)(n-2)}{2kL^2},
\end{align}
\item
The cosmological constant of $(\bar{M} ,\bar{g})$ is $\bar{\Lambda}=0$, that is, $(\bar{M} ,\bar{g})$ is Ricci flat, and
$\bar{g}$ is given by
\begin{align}
ds^2=\frac{1}{L^2}(b-t)^2g_{\alpha\beta}(x^\mu)dx^\alpha
dx^\beta-dt^2, t\neq-b,
\end{align}
where $L\neq0$, and $b\in\mathbb{R}$. Moreover, $(M, g)$ must be an Einstein space with cosmological
constant
\begin{align}
\Lambda=\frac{(n-1)(n-2)}{2L^2}.
\end{align}
\end{enumerate}
\end{theorem}
Putting $n=4$, indicating the five  dimensional spacetime, we obtain
\begin{align}
\Lambda=\frac{3}{L^2}.
\end{align}
This is in agreement with the current observed value of $4D$
cosmological constant provided that $L$ be the current size of the
universe, namely $10^{28}$ meters. Such a good agreement between the
theoretical value of cosmological constant obtained in this higher
dimensional model of warped spaces and the experimental value of
cosmological constant may account for the possible higher
dimensional origin of the cosmological constant.
\section{Conclusion}
 In this article, using the formalism of Bejancu {\it et al}, we have
studied Einstein's equation in $(m+n)D$  and $(1+n)D$ warped spaces
$(\bar{M},\bar{g})$ where $\bar{M}=M_1\times_f M_2$ is the product
manifold $M_1\times M_2$ endowed with the pseudo-Riemannian metric
$\bar{g} = \pi^*g + (f\circ\pi)^2\sigma^*h$. We have classified all such
spaces satisfying Einstein equations
$\bar{G}=-\bar{\Lambda}\bar{g}$. We have shown that the warping
function $f$ can determine both the cosmological constants
$\bar{\Lambda}$, and ${\Lambda}$ appearing in the induced Einstein
equations ${G}=-{\Lambda}{h}$ on $(M_2, h)$. Moreover, we have
discussed on the origin of the $4D$ cosmological constant as an
emergent effect of higher dimensional warped spaces and confronted
its numerical value with its observed value.

\end{document}